\documentclass{amsart}
\usepackage[letterpaper,margin=2.2cm,includehead,headheight=12pt,headsep=14pt]{geometry}

\usepackage{amsmath,amssymb,amscd,mathtools}
\usepackage{enumitem,microtype}
\usepackage{xcolor}
\usepackage[colorlinks=true,linkcolor=blue,citecolor=green!45!black,urlcolor=blue]{hyperref}

\newtheorem{theorem}{Theorem}[section]
\newtheorem{proposition}[theorem]{Proposition}
\newtheorem{lemma}[theorem]{Lemma}
\newtheorem{corollary}[theorem]{Corollary}
\newtheorem{remark}[theorem]{Remark}

\newcommand{\C}{\mathbb C}
\newcommand{\Gm}{\mathbb G_m}
\newcommand{\Db}{\mathrm D^{\mathrm b}}
\newcommand{\Dsg}{\mathrm D_{\mathrm{sg}}}
\newcommand{\Perf}{\operatorname{Perf}}
\newcommand{\MF}{\operatorname{MF}}
\newcommand{\HH}{\operatorname{HH}}
\newcommand{\HS}{\operatorname{HS}}
\newcommand{\Jac}{\operatorname{Jac}}
\newcommand{\pretr}{\operatorname{pretr}}
\newcommand{\Spec}{\operatorname{Spec}}
\newcommand{\Hom}{\operatorname{Hom}}
\newcommand{\RHom}{\operatorname{RHom}}
\newcommand{\id}{\operatorname{id}}
\newcommand{\Mor}{\mathrm{Mor}}

\newcommand{\gr}{\mathrm{gr}}
\newcommand{\dgcat}{\mathrm{dg}}
\newcommand{\m}{\mathfrak m}
\title[Categorical Torelli via singularity category]
{Categorical Torelli theorem for (weighted) hypersurfaces\\
via singularity category}

\author{Xun Lin}
\address[Xun Lin]{School of Science and Engineering, The Chinese University of Hong Kong, Shenzhen, Shenzhen 518172, P.~R.~China}
\email{linxun@cuhk.edu.cn}
\author{Jieheng Zeng}
\address[Jieheng Zeng]{School of Mathematics and Statistics, Hunan Normal University, 15 Taozi Road, Changsha 410081, P.~R.~China}
\email{zengjh662@163.com}
\author{Shizhuo Zhang}
\address[Shizhuo Zhang]{School of Mathematics, Sun Yat-sen University, Guangzhou 510275, P.~R.~China}
\email{zhangshzh28@mail.sysu.edu.cn}


\hypersetup{
 pdftitle={Categorical Torelli theorem for (weighted) hypersurfaces via singularity category},
 pdfauthor={Xun Lin, Jieheng Zeng, and Shizhuo Zhang},
 pdfkeywords={Categorical Torelli theorem, weighted hypersurface, singularity category, matrix factorization}
}

\begin{document}

\begin{abstract}
We prove categorical Torelli theorems for ordinary and weighted hypersurfaces
using the singularity category of the completed affine cone.  For
Fano hypersurfaces we use the Kuznetsov component, whereas for
hypersurfaces of general type we use the full category of graded matrix
factorizations.  When the
degree is coprime to the sum of the weights, the internal grading shift can be
expressed in terms of the Serre functor and the cohomological shift.  Its dg
orbit is therefore intrinsic and, after taking the perfect hull, agrees with
the dg singularity category of the completed cone.  The formal singularity is
then recovered by the categorical Mather--Yau theorem.  For ordinary homogeneous equations,
the resulting formal equivalence induces a projective linear equivalence;
in the weighted setting the dual torus action determines the grading of the
Tjurina algebra and hence the weighted homogeneous equation.  The argument
gives a uniform proof for Fano hypersurfaces under the stated numerical
assumptions and proves new categorical Torelli theorems for hypersurfaces of
general type.
\end{abstract}

\maketitle

\section{Introduction}

\subsection{The categorical Torelli problem}
The bounded derived category is among the strongest homological invariants of
a smooth projective variety.  In particular, varieties with ample canonical
or anticanonical bundle can be reconstructed from their derived categories
\cite[Theorem~2.5]{BondalOrlov}.  For a Fano variety, however, the part which varies in
moduli is often contained in a proper admissible subcategory.  When a
semiorthogonal decomposition has the form
\[
 \Db(X)=\langle \mathcal{K}u(X),E_1,\ldots,E_s\rangle,
\]
with the \(E_i\) exceptional, the admissible subcategory
\(\mathcal{K}u(X)\) is called the Kuznetsov component.  The categorical
Torelli problem asks whether an
equivalence
\[
 \mathcal{K}u(X)\simeq\mathcal{K}u(X')
\]
forces \(X\cong X'\), or which additional structure on the two categories is
needed for such a conclusion.

Several forms of this question have been settled.  The cubic threefold is
reconstructed in~\cite{BernardaraMacriMehrotraStellari}, and a
categorical Torelli theorem for equivalences preserving the rotation functor
is established for a broad class of Fano hypersurfaces
in~\cite[Theorem~1.2]{Pirozhkov}.  Approaches through stability conditions,
Brill--Noether loci, and Hodge-theoretic structures may be found in
\cite{FeyzbakhshLiuZhang,JacovskisLiuZhang,JacovskisLinLiuZhang,BayerPerry}.
For hypersurfaces, Hochschild cohomology is related to the Jacobian algebra in
\cite{HuybrechtsRennemo}.  The Hochschild--Serre algebra used in
\cite{LinZhang,LinRennemoZhang} refines this relation and recovers suitable
graded pieces of the Jacobian algebra directly from the categorical data.

The purpose of the present paper is to give a different proof, based on the
singularity of the affine cone.  The ordinary and weighted theorems, both for
Fano hypersurfaces and for hypersurfaces of general type, reduce to the same local
reconstruction theorem, although they use different categories.  For a Fano hypersurface, the
Kuznetsov component is equivalent to the full graded singularity category by
\cite[Theorems~2.5(i) and~3.10]{Orlov} and
\cite[Theorem~6.13]{BFK}, whereas for a hypersurface of general
type the bounded derived category is only an admissible component of that
category~\cite[Theorem~2.5(ii) and Remark~3.12]{Orlov}.  Accordingly, for hypersurfaces
of general type the theorems below are formulated for the full graded
matrix-factorization category.  The common argument and the separate final
reconstruction steps are explained after the main results.

\subsection{Main results}
Throughout the paper the ground field is \(\C\), and \(\Gm\) denotes the
multiplicative group.  Fix positive weights
\(w_0,\ldots,w_{n+1}\), and introduce the graded polynomial ring and the sum
of the weights by
\[
 P=\C[x_0,\ldots,x_{n+1}],\qquad \deg x_i=w_i,
 \qquad W=\sum_{i=0}^{n+1}w_i.
\]
The vector space of weighted homogeneous polynomials of degree \(d\) is denoted
by \(P_d\).  Let \(\m=(x_0,\ldots,x_{n+1})\), and put
\(\widehat P=\C[[x_0,\ldots,x_{n+1}]]\).  We denote the smooth weighted
projective stack by \(\mathbb P(w_0,\ldots,w_{n+1})\), and write
\(x=(x_0,\ldots,x_{n+1})^{\mathsf t}\) for the coordinate column.  For
\(f\in P_d\), write
\[
 X_f=V(f)\subset\mathbb P(w_0,\ldots,w_{n+1}),\qquad
 R_f=P/(f),\qquad
 A_f=\widehat P/(f),
\]
where \(A_f\) is the completion of \(R_f\) at the homogeneous maximal ideal.
The presentation is called a linear cone if \(f\notin\m^2\), or equivalently
if a variable can be eliminated from the equation by a weight-preserving
change of coordinates.  We write \(\MF^{\gr,\dgcat}(P,f)\) for the dg category
of graded matrix factorizations and \(\Dsg^{\gr,\dgcat}(R_f)\) for the graded dg
singularity category; the ungraded dg singularity category of \(A_f\) is
denoted by \(\Dsg^{\dgcat}(A_f)\).  The full matrix-factorization category and
its internal grading shift will be denoted by
\[
 \mathcal C_f=\MF^{\gr,\dgcat}(P,f),\qquad \rho_f=(1).
\]
When all weights are equal to one, one has \(W=n+2\).
By an enhanced equivalence we mean a \(\C\)-linear dg Morita equivalence
between the standard smooth proper dg enhancements; the notation
\(\simeq_{\Mor}\) will always have this meaning.  For later use, the right
orthogonal to the structure sheaf is
\[
 \mathcal O_X^\perp
 =\{E\in\Db(X)\mid\RHom(\mathcal O_X,E)=0\}.
\]

For Fano hypersurfaces, we first state the result for equivalences that
intertwine the rotation functors.  Under the
equivalence in \cite[Theorems~2.5(i) and~3.10]{Orlov}, the rotation functor
\(\rho_f\) on the Kuznetsov component corresponds to the internal grading
shift \((1)\) on \(\mathcal C_f\).

\begin{theorem}
\label{thm:rotation-reconstruction}
\cite[Theorem~1.7]{LinZhang}.
Let \(f\) and \(g\) define smooth Fano hypersurfaces of the same degree
\(d\ge2\) in the same ordinary projective space.  Suppose that there is an
enhanced equivalence
\[
 \Phi:\mathcal{K}u(X_f)_{\dgcat}\simeq_{\Mor}\mathcal{K}u(X_g)_{\dgcat}
\]
together with an isomorphism of enhanced functors
\(\Phi\rho_f\cong\rho_g\Phi\).  Then \(X_f\cong X_g\) by a projective linear
transformation.  Moreover, the category and its rotation determine the
completed local isomorphism class of the vertex of the affine cone.
\end{theorem}

The reconstruction statement in Theorem~\ref{thm:rotation-reconstruction}
is already known.  It is established in
\cite[Theorem~1.2]{Pirozhkov} for \(d>3\), or \(d=3\) and
\(n>3\), where \(n\) is the dimension of the hypersurface.
The statement for all smooth Fano hypersurfaces is proved in
\cite[Theorem~1.7]{LinZhang}.  We give a new proof using the singularity
category of the completed affine cone.

\begin{corollary}
\label{thm:fano-ordinary}
Let \(f,g\in\C[x_0,\ldots,x_{n+1}]_d\) define smooth Fano hypersurfaces
\(X_f,X_g\subset\mathbb P^{n+1}\), where \(2\le d<n+2\).  Assume
\(\gcd(d,n+2)=1\).  Every enhanced equivalence
\[
 \mathcal{K}u(X_f)_{\dgcat}\simeq_{\Mor}\mathcal{K}u(X_g)_{\dgcat}
\]
is induced at the level of equations by a projective linear transformation:
there are \(L\in\operatorname{GL}_{n+2}(\C)\) and \(c\in\C^*\) such that
\(g(Lx)=cf(x)\).  Consequently \(X_f\cong X_g\).
\end{corollary}

Weighted hypersurfaces are regarded as hypersurfaces in the smooth weighted
projective stack.  Thus quasismoothness refers to the punctured affine cone;
the associated coarse weighted projective variety may still be singular.

\begin{theorem}\label{thm:fano-weighted}
Fix positive weights \(w_0,\ldots,w_{n+1}\) and a degree \(d<W\) with
\(\gcd(d,W)=1\).  Let \(f,g\in P_d\) define quasismooth hypersurface stacks
\(X_f,X_g\subset\mathbb P(w_0,\ldots,w_{n+1})\), and assume that neither
presentation is a linear cone.  Every enhanced equivalence
\[
 \mathcal{K}u(X_f)_{\dgcat}\simeq_{\Mor}\mathcal{K}u(X_g)_{\dgcat}
\]
implies that a weight-preserving polynomial automorphism sends one equation to
a nonzero scalar multiple of the other.  Hence \(X_f\cong X_g\) as weighted
hypersurface stacks, and their coarse spaces are isomorphic.
\end{theorem}

The specialization \(d=W-1\) has an additional interpretation.  Perry's
conjecture, in the form recorded and studied in
\cite{LinZhangOrthogonal}, predicts that a smooth Fano variety \(X\) is
determined by the enhanced right orthogonal \(\mathcal O_X^\perp\).  For an
index-one Fano hypersurface, the standard exceptional block has length
\(W-d=1\), and the semiorthogonal decomposition becomes
\[
 \Db(X)=\langle\mathcal{K}u(X),\mathcal O_X\rangle,
 \qquad \mathcal{K}u(X)=\mathcal O_X^\perp.
\]
Consequently, categorical Torelli for an index-one Fano hypersurface is
precisely the corresponding instance of Perry's conjecture.
Theorem~\ref{thm:fano-weighted} therefore proves the analogous statement for every quasismooth weighted
index-one hypersurface stack whose defining equation lies in \(\m^2\).  The next two corollaries verify the conjecture for
ordinary index-one hypersurfaces and for the indicated index-one double covers.

\begin{corollary}\label{cor:index-one}
Let \(X_f,X_g\subset\mathbb P^{n+1}\) be smooth hypersurfaces of degree
\(n+1\).  Then
\[
 \mathcal O_{X_f}^{\perp}\simeq_{\Mor}\mathcal O_{X_g}^{\perp}
 \quad\Longrightarrow\quad X_f\cong X_g.
\]
The orthogonal category also determines the completed cone singularity
\(A_f\).
\end{corollary}

\begin{corollary}\label{cor:double-covers}
Let \(n\ge2\), and let \(X_f,X_g\) be smooth double covers of \(\mathbb P^n\)
branched along smooth hypersurfaces of degree \(2n\).  An enhanced equivalence
of their Kuznetsov components implies \(X_f\cong X_g\) and identifies their
branch divisors by a projective linear transformation.
\end{corollary}

For hypersurfaces of general type, we formulate the result for the full
category \(\mathcal C_f\).

\begin{theorem}
\label{thm:general-ordinary}
Let \(d>n+2\), assume \(\gcd(d,n+2)=1\), and let
\(f,g\in\C[x_0,\ldots,x_{n+1}]_d\) define smooth hypersurfaces
\(X_f,X_g\subset\mathbb P^{n+1}\).  If
\[
 \MF^{\gr,\dgcat}(\C[x_0,\ldots,x_{n+1}],f)
 \simeq_{\Mor}
 \MF^{\gr,\dgcat}(\C[x_0,\ldots,x_{n+1}],g),
\]
then there are \(L\in\operatorname{GL}_{n+2}(\C)\) and \(c\in\C^*\) such
that \(g(Lx)=cf(x)\).  Consequently \(X_f\cong X_g\) as polarized varieties.
\end{theorem}

\begin{theorem}
\label{thm:general-weighted}
Fix the weights and a degree \(d>W\) satisfying \(\gcd(d,W)=1\).  Let
\(f,g\in P_d\) define quasismooth hypersurface stacks in
\(\mathbb P(w_0,\ldots,w_{n+1})\), and assume that neither presentation is a
linear cone.  If
\[
 \MF^{\gr,\dgcat}(P,f)\simeq_{\Mor}\MF^{\gr,\dgcat}(P,g),
\]
then a weight-preserving polynomial automorphism carries \(g\) to a nonzero
scalar multiple of \(f\).  Hence \(X_f\cong X_g\) as weighted hypersurface
stacks, and their coarse spaces are isomorphic.
\end{theorem}

\subsection{Strategy of the proof}
We now describe the argument used for ordinary and weighted hypersurfaces,
both when \(d<W\) and when \(d>W\), retaining the notation introduced above.
For a weighted homogeneous polynomial \(f\in P_d\), the standard equivalence
gives
\[
 \mathcal C_f=\MF^{\gr,\dgcat}(P,f)
 \simeq_{\Mor}\Dsg^{\gr,\dgcat}(R_f).
\]
Quasismoothness ensures that the cone is smooth away from its vertex; the assumption \(f\in\m^2\) makes the vertex singular.  Thus \(A_f\) is an
isolated complete local hypersurface singularity, and \(\mathcal C_f\) is
smooth and proper.  When \(d<W\), the category \(\mathcal C_f\) is identified with
\(\mathcal{K}u(X_f)\) by
\cite[Theorems~2.5(i) and~3.10]{Orlov} and
\cite[Theorem~6.13]{BFK}.  When \(d>W\), we use the full category \(\mathcal C_f\).

Let \(S_{\mathcal C_f}\) denote the Serre functor of \(\mathcal C_f\), and let
\(\omega_{R_f}\) denote the graded canonical module of \(R_f\).  The internal
shift satisfies
\[
 \rho_f^d\simeq[2],\qquad
 S_{\mathcal C_f}\simeq\rho_f^{d-W}[n].
\]
The first is the two-periodicity of a degree-\(d\) graded matrix factorization;
the second follows from graded Auslander--Reiten duality and
\(\omega_{R_f}\cong R_f(d-W)\).  Under the assumption \(\gcd(d,W)=1\), choose
\(u,v\in\mathbb Z\) such that \(u(d-W)+vd=1\).  Then
\[
 \rho_f\simeq S_{\mathcal C_f}^{u}[2v-un].
\]
The grading shift is consequently determined by the Serre functor and the
cohomological shift.  An enhanced equivalence
\(F:\mathcal C_f\simeq\mathcal C_g\) therefore admits a canonical
intertwining isomorphism
\(F\rho_f\simeq\rho_gF\).

We next recover the singularity category of the completed cone by forgetting
the internal grading.  Since an object and its grading shifts have the same
underlying ungraded object, the forgetful functor factors through the dg orbit
category \(\mathcal C_f/\rho_f\).  Passing to its perfect dg hull and
completing at the vertex gives a Morita equivalence
\[
 \Perf_{\dgcat}\bigl(\mathcal C_f/\rho_f\bigr)
 \simeq_{\Mor}\Dsg^{\dgcat}(A_f).
\]
Here \(\Perf_{\dgcat}(-)\) denotes the perfect dg hull, obtained by adjoining
shifts, cones, and direct summands.  The corresponding equivalence of
triangulated categories is established in \cite[Proposition~A.8]{KMVdB}.
To obtain this Morita equivalence, we use the functor induced by forgetting the
grading.  The results of \cite[Appendix~A]{KMVdB} show that it is fully
faithful on the dg orbit and that its image generates the ungraded
singularity category after adjoining cones and direct summands.
The passage to the completed local ring is justified by
\cite[Theorem~5.7]{Dyckerhoff}; see Proposition~\ref{prop:orbit-completion}
for the details.  Thus an enhanced equivalence compatible with the grading
shifts induces
\[
 \Dsg^{\dgcat}(A_f)\simeq_{\Mor}\Dsg^{\dgcat}(A_g).
\]

The local reconstruction theorem~\cite[Theorem~5.9]{HuaKeller} identifies the
degree-zero Hochschild cohomology as
\[
 \HH^0\!\left(\Dsg^{\dgcat}(A_f)\right)
 \cong
 T_f:=\frac{\C[[x_0,\ldots,x_{n+1}]]}
 {(f,\partial_0f,\ldots,\partial_{n+1}f)}
\]
and reconstructs \(A_f\) from its dg singularity category when the dimension
is fixed.  This dimension hypothesis holds because the hypersurfaces
compared in our theorems have the same ambient dimension.

It remains to pass from the formal germ to the projective hypersurface.  This
last step depends on the grading and differs for ordinary and weighted
equations.

\medskip
\noindent\emph{Ordinary Fano hypersurfaces.}
The equivalence \(\mathcal{K}u(X_f)\simeq\mathcal{K}u(X_g)\) is transported to
an equivalence \(\mathcal C_f\simeq\mathcal C_g\).  Coprimality makes the
rotation intrinsic, so the orbit construction recovers \(A_f\cong A_g\).  A formal contact equivalence between two homogeneous
equations has an invertible linear part.  Comparing the terms of degree
\(d\) gives
\(g(Lx)=cf(x)\).  Hence \(X_f\cong X_g\).

\medskip
\noindent\emph{Weighted Fano hypersurfaces.}
The construction relating \(\mathcal{K}u(X_f)\) to \(A_f\) is compatible
with the torus actions.
The dual \(\Gm\)-action on the orbit becomes the Euler action on the completed
cone.  Naturality of the preceding identification of \(\HH^0\) therefore
recovers \(T_f\cong T_g\) as graded algebras.  The normalized Euler identity
\[
 \deg(f)\,f=\sum_i w_i x_i\partial_i f
\]
identifies the Tjurina and Jacobian algebras, and
\cite[Corollary~3]{Ahmed} gives a weight-preserving change of coordinates.

\medskip
\noindent\emph{Ordinary hypersurfaces of general type.}
Here the theorem begins with an equivalence of the full categories
\(\mathcal C_f\simeq\mathcal C_g\).  The same Serre-functor calculation makes
\(\rho_f\) intrinsic, and the same orbit comparison gives \(A_f\cong A_g\).
The homogeneous initial-term argument then yields a projective linear
isomorphism \(X_f\cong X_g\).

\medskip
\noindent\emph{Weighted hypersurfaces of general type.}
The full matrix-factorization equivalence is again passed to the dg orbit, now
with its dual \(\Gm\)-action.  The resulting isomorphism of graded Tjurina
algebras, together with the Euler identity and weighted homogeneous
reconstruction, gives a weight-preserving automorphism carrying one equation
to a scalar multiple of the other.

\subsection{Acknowledgements}
Shizhuo Zhang acknowledges support from the start-up funding of
Sun Yat-sen University (grant no.~34000-12256019).

\subsection{AI disclosure}
The categorical Torelli results for Fano hypersurfaces under the numerical
assumptions of this paper were already established in
\cite[Theorem~1.4]{LinZhang} for ordinary hypersurfaces and in
\cite[Theorem~1.1]{Lin2026} for weighted hypersurfaces.  The new proofs
presented here were developed by the authors.  They follow the suggestion in
\cite[Section~1.2, p.~79]{HuybrechtsRennemo}, entitled ``Graded matrix
factorizations,'' to pass from graded matrix factorizations to an orbit by
the internal grading shift and recover the defining equation from Hochschild
cohomology.  They were also inspired by the earlier work of Xiaojun Chen and
Jieheng Zeng, \emph{Singularity categories and singular loci of certain
abelian quotient singularities}~\cite{ChenZeng}, on recovering geometric
information from singularity categories.

The proofs for hypersurfaces of general type follow the same argument,
with the full graded matrix-factorization category in place of the Kuznetsov
component, and were also developed by the authors.  AI tools assisted with
checking details and reviewing the manuscript, including identifying and
correcting technical errors and typographical mistakes.  The authors are
responsible for the proofs and the final text.

\subsection{Organization and conventions}
Section~\ref{sec:preliminaries} recalls the required material on dg categories,
Hochschild--Serre algebras, matrix factorizations, singularity categories, and
dg orbits.  Section~\ref{sec:common} proves the reconstruction theorem for ordinary
hypersurfaces.  The Euler-equivariant weighted refinement is proved in
Section~\ref{sec:weighted-engine}.  Fano hypersurfaces are treated in
Section~\ref{sec:fano}, and hypersurfaces of general type are treated in
Section~\ref{sec:general-type}.  The comparison with the Hochschild--Serre
method is given in Section~\ref{sec:comparison}.

All categories are \(\C\)-linear.  We use \(\cong\) for isomorphisms of
varieties, stacks, rings, and functors, and \(\simeq\) for equivalences of
categories.  Dg orbits are always followed by pretriangulated and idempotent
completion.  Completed rings are completed at the homogeneous maximal ideal.
For a graded module we use the convention
\[
 M(1)_q=M_{q+1}.
\]
Thus a degree-zero morphism \(M\to N(j)\) is, after forgetting the grading, a
homogeneous morphism of degree \(j\).  Whenever an equivalence is said to
intertwine two autoequivalences, the intertwining isomorphism is part of the
enhanced data and its positive and negative iterates are formed using the
standard cocycle convention.  We replace dg categories by cofibrant
Morita-fibrant models when forming dg orbits; this does not change any perfect
dg hull occurring below.

\section{Preliminaries}\label{sec:preliminaries}

\subsection{Enhancements and semiorthogonal decompositions}
For a smooth projective variety or a smooth proper Deligne--Mumford stack \(X\),
we write \(\Db(X)\) for the bounded derived category of coherent sheaves and
use its standard dg enhancement.  A sequence of full triangulated
subcategories is a semiorthogonal decomposition
\[
 \mathcal D=\langle\mathcal A_1,\ldots,\mathcal A_m\rangle
\]
if \(\Hom(\mathcal A_j,\mathcal A_i)=0\) for \(j>i\) and the subcategories
generate \(\mathcal D\).  A subcategory is admissible if its inclusion has both
adjoints.  Admissible subcategories of \(\Db(X)\) inherit smooth proper dg
enhancements and therefore Serre functors.

An object \(E\) is exceptional if \(\RHom(E,E)\cong\C\).  Left mutation through
\(E\) is defined by the functorial triangle
\begin{equation}\label{eq:left-mutation}
 \RHom(E,F)\otimes E\longrightarrow F\longrightarrow \mathbf L_E(F).
\end{equation}
For a Fano variety, \(\mathcal O_X\) is exceptional and
\[
 \Db(X)=\langle\mathcal O_X,\mathcal O_X^\perp\rangle,
 \qquad
 \mathcal O_X^\perp=\{F\mid\RHom(\mathcal O_X,F)=0\}.
\]
If the standard exceptional block has more than one object, the right
orthogonal \(\mathcal O_X^\perp\) need not coincide with the Kuznetsov
component.  The two
categories agree for the index-one hypersurfaces considered below.

A dg functor \(F:\mathcal A\to\mathcal B\) is a quasi-equivalence if the map
\[
 \mathcal A(E,E')\longrightarrow\mathcal B(FE,FE')
\]
is a quasi-isomorphism for every pair of objects \(E,E'\in\mathcal A\), and
the functor \(H^0(F):H^0(\mathcal A)\to H^0(\mathcal B)\) is essentially
surjective.  The first condition implies that \(H^0(F)\) is fully faithful,
so the second condition is equivalently the assertion that \(H^0(F)\) is an
equivalence.

A dg functor \(F:\mathcal A\to\mathcal B\) is a Morita equivalence if
restriction of scalars induces an equivalence
\[
 F^*:\mathrm D(\mathcal B)\xrightarrow{\sim}\mathrm D(\mathcal A)
\]
between the derived categories of right dg modules.  Equivalently, derived
extension of scalars restricts to an equivalence
\[
 -\otimes_{\mathcal A}^{\mathbf L}\mathcal B:
 \operatorname{per}(\mathcal A)\xrightarrow{\sim}
 \operatorname{per}(\mathcal B).
\]
Every quasi-equivalence is a Morita equivalence, but the converse is false in
general.  An isomorphism in the Morita homotopy category is represented, after
passing to pretriangulated idempotent-complete Morita-fibrant replacements, by
a zigzag of quasi-equivalences.  Whenever a dg orbit is formed below, we choose
cofibrant representatives for this zigzag and for the relevant
autoequivalences.  These conventions are those of
\cite[Sections~2.2--2.3, 3.1, 3.8, and~4.1]{KellerDG} and
\cite[Definition~2.1 and Sections~3 and~7]{Toen}.

A Morita equivalence between smooth proper dg categories intertwines their
Serre functors, by uniqueness of the representing Serre bimodule.

A triangulated equivalence is sufficient for our arguments only when a dg lift
is known, for example for an equivalence of Fourier--Mukai type.  The dg enhancement
is required because dg orbits and Hochschild cohomology are not invariants
of a triangulated category without a specified enhancement; see
\cite{Toen,KellerDG}.

\subsection{Serre functors, rotation, and the Hochschild--Serre algebra}
Let \(\mathcal C\) be a smooth proper dg category.  Its Serre functor is
characterized by functorial isomorphisms
\[
 \Hom_{H^0(\mathcal C)}(E,F)^\vee
 \cong \Hom_{H^0(\mathcal C)}(F,S_{\mathcal C}E).
\]
Every enhanced equivalence intertwines Serre functors.  We use the convention
\begin{equation}\label{eq:HS-definition}
 \HS^{p,q}(\mathcal C)=
 \Hom_{\operatorname{Fun}(\mathcal C,\mathcal C)}
 \bigl(\id_{\mathcal C},S_{\mathcal C}^{q}[p]\bigr),
 \qquad
 \HS(\mathcal C)=\bigoplus_{p,q\in\mathbb Z}\HS^{p,q}(\mathcal C),
\end{equation}
with multiplication given by twisted composition.  If
\(u:\id\to S^q[p]\) and \(v:\id\to S^{q'}[p']\), then
\begin{equation}\label{eq:HS-product}
 u\cdot v:\id\xrightarrow{v}S^{q'}[p']
 \xrightarrow{S^{q'}[p'](u)}S^{q+q'}[p+p'].
\end{equation}
This convention fixes the order in the product.  An enhanced equivalence
induces an isomorphism of bigraded algebras, not merely a bigraded vector-space
isomorphism.

For an ordinary degree-\(d\) Fano hypersurface of index \(a\), define the
rotation on its Kuznetsov component by
\begin{equation}\label{eq:geometric-rotation}
 \rho(F)=\mathbf L_{\mathcal O_X}\bigl(F\otimes\mathcal O_X(1)\bigr),
\end{equation}
with the analogous definition on the weighted hypersurface stack.  Under the
equivalence in~\cite[Theorems~2.5(i) and~3.10]{Orlov} it becomes the internal
grading shift \((1)\).  In
general, the choice of \(\rho\) is additional data.  Under the coprimality hypotheses in this paper it is an
explicit expression in the Serre functor and the cohomological shift and is
therefore intrinsic.

\subsection{Graded singularity categories and matrix factorizations}
The definitions of dg categories, dg functors, dg modules, dg quotients,
quasi-equivalences, and Morita equivalences used in this subsection may be
found in \cite[Sections~2.2--3.1 and~3.8]{KellerDG} and
\cite[Definition~2.1 and Sections~3 and~7]{Toen}.  We use cohomological
grading throughout.

For a Noetherian ring \(R\), the singularity category is
\[
 \Dsg(R)=\Db(\operatorname{mod}R)/\Perf(R).
\]
The Drinfeld quotient gives its canonical dg enhancement
\(\Dsg^{\dgcat}(R)\).  For a graded ring one similarly has
\(\Dsg^{\gr,\dgcat}(R)\) and the internal shift \((1)\).

Let \(B\) be regular and \(f\in B\).  A matrix factorization is a
\(\mathbb Z/2\)-graded finite projective \(B\)-module with an odd differential
whose square is multiplication by \(f\).  If \(f\) has degree \(d\), a graded
factorization may be displayed as
\[
 E_0\xrightarrow{\delta_0}E_1\xrightarrow{\delta_1}E_0(d).
\]
Shifting twice applies the internal shift by \(d\), so
\begin{equation}\label{eq:periodicity}
 (d)\simeq[2].
\end{equation}
The homotopy category of matrix factorizations is identified with the
hypersurface singularity category in
\cite{Buchweitz,Eisenbud,Orlov}.  The dg lift in
\cite[Theorem~2.49]{BRTV} gives
\begin{equation}\label{eq:mf-sg}
 \MF^{\gr,\dgcat}(P,f)\simeq_{\Mor}\Dsg^{\gr,\dgcat}(R_f).
\end{equation}

Recall that
\[
 P=\C[x_0,\ldots,x_{n+1}],\qquad \deg x_i=w_i,
 \qquad W=\sum_{i=0}^{n+1}w_i.
\]
The canonical bundle of the weighted projective stack is
\(\omega_{\mathbb P(w_0,\ldots,w_{n+1})}\cong
\mathcal O_{\mathbb P(w_0,\ldots,w_{n+1})}(-W)\); equivalently, the graded
canonical module of the polynomial ring \(P\) is \(P(-W)\).  The hypersurface
resolution
\[
 0\longrightarrow P(-d)\xrightarrow{\cdot f}P\longrightarrow R_f
 \longrightarrow0
\]
therefore gives
\begin{equation}\label{eq:canonical-module}
 \omega_{R_f}\cong R_f(d-W).
\end{equation}
Since \(R_f\) has Krull dimension \(n+1\), graded Auslander--Reiten duality
gives
\begin{equation}\label{eq:serre-full}
 S_{\mathcal C_f}\simeq \rho_f^{\,d-W}[n].
\end{equation}
The same formula appears immediately before formula~(5.4) in
\cite[Section~5.3]{KMVdB}; for the graded matrix-factorization convention,
compare~\cite[Proposition~3.7]{LinZhang}.

\begin{proposition}\label{prop:intrinsic-shift}
Assume \(\gcd(d,W)=1\).  Choose \(u,v\in\mathbb Z\) satisfying
\begin{equation}\label{eq:bezout}
 u(d-W)+vd=1.
\end{equation}
Then there is an isomorphism of enhanced functors
\begin{equation}\label{eq:intrinsic-shift}
 \rho_f\simeq S_{\mathcal C_f}^{u}[2v-un].
\end{equation}
Consequently every enhanced equivalence between full graded
matrix-factorization categories with the same weights and degree automatically
intertwines their internal shifts.
\end{proposition}

\begin{proof}
Put \(b=d-W\).  The functors \(\rho_f\), \([1]\), and
\(S_{\mathcal C_f}\) commute up to the canonical dg isomorphisms.  Hence the
identities \(S_{\mathcal C_f}\simeq\rho_f^b[n]\) and
\(\rho_f^d\simeq[2]\) give, also for negative \(u\) and \(v\),
\begin{align*}
 S_{\mathcal C_f}^{u}[2v-un]
 &\simeq (\rho_f^b[n])^u[2v-un]  \\
 &\simeq \rho_f^{ub}[2v]
  \simeq \rho_f^{ub}(\rho_f^d)^v
  \simeq \rho_f^{ub+vd}=\rho_f.
\end{align*}
Here inverse powers are taken in the group of dg autoequivalences of
\(\mathcal C_f\).  The equality \(ub+vd=1\) is available because
\(\gcd(d,b)=\gcd(d,W)=1\).

Let \(F:\mathcal C_f\to\mathcal C_g\) be an enhanced equivalence and let
\(G\) be a quasi-inverse.  Since the Serre bimodule is characterized by
\[
 \RHom_{\mathcal C_f}(E,E')^\vee
 \simeq
 \RHom_{\mathcal C_f}(E',S_{\mathcal C_f}E),
\]
transport by \(F\) gives a canonical isomorphism
\(F S_{\mathcal C_f}\xrightarrow{\sim}S_{\mathcal C_g}F\).  Taking the
\(u\)-th power and composing with the strict compatibility of a dg functor with
cohomological shifts yields
\[
 F\rho_f
 \simeq F S_{\mathcal C_f}^{u}[2v-un]
 \simeq S_{\mathcal C_g}^{u}[2v-un]F
 \simeq \rho_gF.
\]
This is the required isomorphism of functors.
\end{proof}

\begin{remark}
If \(e=\gcd(d,W)>1\), the same calculation recovers \(\rho_f^e\), not a
preferred \(e\)-th root.  All reconstruction results below remain true for an
equivalence explicitly equipped with \(F\rho_f\cong\rho_gF\); coprimality is
used only to make this compatibility automatic.
\end{remark}

\subsection{Dg orbits and perfect hulls}
Let \(\mathcal C\) be a dg category and \(T\) a dg autoequivalence.  After
choosing cofibrant models, the dg orbit has the same objects and morphism
complexes
\begin{equation}\label{eq:dg-orbit}
 \Hom_{\mathcal C/T}(E,F)=
 \bigoplus_{j\in\mathbb Z}\Hom_{\mathcal C}(E,T^jF),
\end{equation}
with composition
\begin{equation}\label{eq:orbit-composition}
 (k_j\star h_i):=T^i(k_j)\circ h_i
 \in\Hom_{\mathcal C}(E,T^{i+j}G)
\end{equation}
for \(h_i:E\to T^iF\) and \(k_j:F\to T^jG\).  The differential is the direct
sum differential, and \eqref{eq:orbit-composition} is a chain map because
\(T\) is a dg functor.  We write
\[
 \Perf_{\dgcat}(\mathcal C/T):=\pretr(\mathcal C/T)^{\natural}
\]
for its pretriangulated idempotent-complete hull.

\begin{lemma}\label{lem:shifted-H0}
Let \(F:\mathcal C\to\mathcal D\) be a dg functor between dg categories with
compatible shifts.  If \(H^0(F)\) is fully faithful on all shifted objects, then
\(F\) is quasi-fully faithful.  If \(H^0(\mathcal D)\) is classically generated
by the image, the induced functor on perfect dg hulls is a quasi-equivalence.
\end{lemma}

\begin{proof}
For \(E,E'\in\mathcal C\) and \(p\in\mathbb Z\), compatibility with shifts
gives a commutative square
\[
\begin{CD}
 H^p\Hom_{\mathcal C}(E,E') @>{H^p(F_{E,E'})}>>
 H^p\Hom_{\mathcal D}(FE,FE')\\
 @V{\sim}VV @VV{\sim}V\\
 \Hom_{H^0(\mathcal C)}(E,E'[p]) @>{H^0(F)}>>
 \Hom_{H^0(\mathcal D)}(FE,FE'[p]).
\end{CD}
\]
The bottom arrow is bijective by hypothesis.  Thus
\(F_{E,E'}\) induces an isomorphism in every cohomological degree and is a
quasi-isomorphism.  Hence \(F\) is quasi-fully faithful.

The induced functor \(\pretr(F)\) is again quasi-fully faithful because its
morphism complexes are obtained from the original ones by finite sums, shifts,
and cones.  Passing to the Karoubi envelope preserves this property.  If the
image of \(H^0(F)\) classically generates \(H^0(\mathcal D)\), every object of
the pretriangulated idempotent-complete hull of \(\mathcal D\) is a direct
summand of a finite iterated cone of shifts of objects \(FE\).  Therefore
\(H^0(\pretr(F)^\natural)\) is essentially surjective, and
\(\pretr(F)^\natural\) is a quasi-equivalence.
\end{proof}

\begin{proposition}
\label{prop:orbit-completion}
Assume that \(\Spec R_f\) is smooth away from the homogeneous maximal ideal
\(\m=(x_0,\ldots,x_{n+1})\) and singular at \(\m\).  Then there is a natural
Morita equivalence
\begin{equation}\label{eq:orbit-completion}
 \Perf_{\dgcat}\!\left(
   \Dsg^{\gr,\dgcat}(R_f)/(1)\right)
 \simeq_{\Mor}
 \Dsg^{\dgcat}(A_f).
\end{equation}
\end{proposition}

\begin{proof}
Set \(\mathcal G_f=\Dsg^{\gr,\dgcat}(R_f)\), and denote its dg orbit by
\(\mathcal O_f=\mathcal G_f/(1)\).  Forgetting the grading
defines a dg functor
\[
 U_f:\mathcal O_f\longrightarrow\Dsg^{\dgcat}(R_f).
\]
On objects it sends a graded complex to its underlying complex.  On the
\(j\)-th orbit summand it is the map
\[
 \Hom_{\mathcal G_f}(M,N(j))\longrightarrow
 \Hom_{\Dsg^{\dgcat}(R_f)}(U_fM,U_fN)
\]
obtained by forgetting that the underlying morphism has internal degree \(j\).
Formula~\eqref{eq:orbit-composition} shows that these maps define a dg functor.

For \(M,N\in\mathcal G_f\), it is proved in
\cite[Lemma~A.7]{KMVdB} that the induced map
\begin{equation}\label{eq:KMVdB-Hom}
 \bigoplus_{j\in\mathbb Z}
 \Hom_{H^0(\mathcal G_f)}(M,N(j))
 \xrightarrow{\ \sim\ }
 \Hom_{\Dsg(R_f)}(U_fM,U_fN)
\end{equation}
is bijective.  Replacing \(N\) by \(N[p]\) in
\eqref{eq:KMVdB-Hom} and applying Lemma~\ref{lem:shifted-H0} shows that
\(U_f\) is quasi-fully faithful.

Let
\[
 L_f:\Dsg^{\dgcat}(R_f)\longrightarrow
 \Dsg^{\dgcat}((R_f)_{\m}),\qquad
 C_f:\Dsg^{\dgcat}((R_f)_{\m})\longrightarrow
 \Dsg^{\dgcat}(A_f)
\]
be derived extension of scalars.  Both functors are defined on the standard dg
quotients because a bounded complex of projectives remains perfect after base
change.  Since the singular locus of \(R_f\) is \(\{\m\}\), localization gives
an equivalence on homotopy categories.  Completion is fully faithful and its
image is dense up to direct summands by
\cite[Proposition~A.1]{KMVdB}.  Moreover the residue field
\(k=R_f/\m\), with its natural grading, lies in the image of \(U_f\), and
\(k\) classically generates each local singularity category by
\cite[Proposition~A.2]{KMVdB}.  It follows that the image of
\(C_fL_fU_f\) classically generates \(\Dsg(A_f)\) after taking direct
summands.

Applying Lemma~\ref{lem:shifted-H0} to the composite and then taking
pretriangulated idempotent-complete hulls gives a quasi-equivalence
\[
 \pretr(\mathcal O_f)^\natural
 \xrightarrow{\ \sim\ }
 \Perf_{\dgcat}\bigl(\Dsg^{\dgcat}(A_f)\bigr).
\]
The right-hand dg category is Morita equivalent to
\(\Dsg^{\dgcat}(A_f)\).  For the hypersurfaces considered here, the same final step
may be expressed through matrix factorizations: the perfect hull before
completion is identified with \(\MF^{\dgcat}(\widehat P,f)\) in
\cite[Theorem~5.7]{Dyckerhoff}, and the corresponding homotopy category is
identified with the idempotent completion in
\cite[Lemma~5.6]{Dyckerhoff}.  This proves \eqref{eq:orbit-completion}.
\end{proof}

\begin{remark}
Proposition~A.8 of~\cite{KMVdB} is stated as an equivalence of triangulated
categories.  Proposition~\ref{prop:orbit-completion} does not treat it as a dg
equivalence.  The explicit forgetful functor and the full-faithfulness and
generation results of~\cite[Appendix~A]{KMVdB} are combined with Lemma~\ref{lem:shifted-H0} and the completion theorem
\cite[Theorem~5.7]{Dyckerhoff}.  The perfect hull is essential.
\end{remark}

\section{The common categorical Mather--Yau reduction}\label{sec:common}

\subsection{The categorical Mather--Yau theorem}
Let
\[
 A_f=\C[[x_0,\ldots,x_{n+1}]]/(f)
\]
be an isolated complete local hypersurface singularity.  Its Tjurina algebra is
\[
 T_f=\frac{\C[[x_0,\ldots,x_{n+1}]]}
 {(f,\partial_0f,\ldots,\partial_{n+1}f)}.
\]

\begin{theorem}\label{thm:hua-keller}
By~\cite[Theorem~5.9]{HuaKeller}, whose reconstruction assertion uses
the Mather--Yau theorem~\cite[Theorem, p.~243]{MatherYau}, there is an
algebra isomorphism
\begin{equation}\label{eq:HK}
 \HH^0\!\left(\Dsg^{\dgcat}(A_f)\right)\cong T_f.
\end{equation}
If \(A_g\) is another isolated complete local hypersurface singularity in the
same number of variables and
\(\Dsg^{\dgcat}(A_f)\simeq_{\Mor}\Dsg^{\dgcat}(A_g)\), then
\(A_f\cong A_g\).
\end{theorem}

\begin{theorem}\label{thm:kalck}
By~\cite[Theorem~1.4]{Kalck}, the following classification holds.  Let \(A_f\)
be an isolated complete local hypersurface singularity and let
\(S\) be a complete local commutative Noetherian \(\C\)-algebra.  If their dg
singularity categories are Morita equivalent, then \(S\) is a hypersurface and
the larger-dimensional one of the two equations is obtained from the other by
adding an even number of nondegenerate quadratic variables.  Conversely, such
an even Kn\"orrer stabilization gives a Morita equivalence.
\end{theorem}

In Theorem~\ref{thm:hua-keller}, fixing the number of variables, equivalently
the embedding dimension for a non-linear hypersurface equation, excludes the
even Kn\"orrer ambiguity described in Theorem~\ref{thm:kalck}.  If the Krull
dimensions are equal, no quadratic variables can be added.

\subsection{From a graded category to the completed cone}
\begin{proposition}\label{prop:common-local}
Let \(f\) and \(g\) be weighted homogeneous equations of the same degree in the
same weighted polynomial ring, with isolated cone singularities.  Suppose
\[
 F:\mathcal C_f\simeq_{\Mor}\mathcal C_g
\]
is equipped with an isomorphism \(F\rho_f\cong\rho_gF\).  Then
\begin{equation}\label{eq:local-equivalence}
 \Dsg^{\dgcat}(A_f)\simeq_{\Mor}\Dsg^{\dgcat}(A_g),
 \qquad A_f\cong A_g.
\end{equation}
If \(\gcd(d,W)=1\), the compatibility with \(\rho\) is automatic.
\end{proposition}

\begin{proof}
Choose a dg quasi-equivalence representing \(F\) and an enhanced natural
isomorphism
\(\eta:F\rho_f\xrightarrow{\sim}\rho_gF\).  Define
\(\eta_j:F\rho_f^j\xrightarrow{\sim}\rho_g^jF\) recursively by
\[
 \eta_0=\id_F,\qquad
 \eta_{j+1}=(\rho_g^j\eta)\circ(\eta_j\rho_f)\quad(j\ge0),
\]
and define \(\eta_j\) for \(j<0\) from the inverse of \(\eta\).  The orbit
functor \(F_{\mathrm{orb}}:\mathcal C_f/\rho_f\to
\mathcal C_g/\rho_g\) is given by
\begin{equation}\label{eq:induced-orbit-functor}
 F_{\mathrm{orb}}(h_j)=\eta_{j,E'}\circ F(h_j)
 \quad\text{for }h_j:E\longrightarrow\rho_f^jE'.
\end{equation}
The cocycle identity
\[
 \eta_{i+j}=(\rho_g^i\eta_j)\circ(\eta_i\rho_f^j)
\]
and formula~\eqref{eq:orbit-composition} imply
\(F_{\mathrm{orb}}(k_j\star h_i)=
F_{\mathrm{orb}}(k_j)\star F_{\mathrm{orb}}(h_i)\); hence
\eqref{eq:induced-orbit-functor} is a dg functor.  Applying the same construction
to a quasi-inverse of \(F\) proves that \(F_{\mathrm{orb}}\) is a Morita
equivalence.  Therefore
\[
 \Perf_{\dgcat}(\mathcal C_f/\rho_f)
 \simeq_{\Mor}
 \Perf_{\dgcat}(\mathcal C_g/\rho_g).
\]
Proposition~\ref{prop:orbit-completion}, applied to \(R_f\) and \(R_g\), gives
the first equivalence in \eqref{eq:local-equivalence}.

Both \(A_f\) and \(A_g\) are complete local hypersurface algebras of Krull
dimension \(n+1\) and embedding dimension \(n+2\).  Their singularities are
isolated because the punctured affine cones are smooth.  Thus all hypotheses of
Theorem~\ref{thm:hua-keller} hold, and that theorem gives \(A_f\cong A_g\).
If \(\gcd(d,W)=1\), Proposition~\ref{prop:intrinsic-shift} constructs \(\eta\)
from the isomorphism intertwining the Serre functors, so no compatibility with \(\rho\) has to be
assumed separately.
\end{proof}

\subsection{Reconstruction of ordinary projective hypersurfaces}
\begin{lemma}\label{lem:ordinary-globalization}
Let \(f,g\in\C[x_0,\ldots,x_N]\) be nonzero homogeneous polynomials of the same
ordinary degree \(d\ge2\).  If
\[
 \frac{\C[[x_0,\ldots,x_N]]}{(f)}
 \cong
 \frac{\C[[x_0,\ldots,x_N]]}{(g)}
\]
as complete local \(\C\)-algebras, then there are
\(L\in\operatorname{GL}_{N+1}(\C)\) and \(c\in\C^*\) such that
\begin{equation}\label{eq:ordinary-globalization}
 g(Lx)=cf(x).
\end{equation}
\end{lemma}

\begin{proof}
Put \(Q=\C[[x_0,\ldots,x_N]]\), let \(\mathfrak n=(x_0,\ldots,x_N)\), and
write \(q_f:Q\to Q/(f)\) and \(q_g:Q\to Q/(g)\) for the quotient maps.  Denote
the maximal ideals of the quotient rings by
\[
 \mathfrak m_f=\mathfrak n/(f),\qquad
 \mathfrak m_g=\mathfrak n/(g).
\]
Let
\(\phi:Q/(f)\xrightarrow{\sim}Q/(g)\).  Choose
\(u_i\in\mathfrak n\) satisfying
\(q_g(u_i)=\phi(q_f(x_i))\), and define the continuous homomorphism
\[
 U:Q\longrightarrow Q,\qquad U(x_i)=u_i.
\]
It satisfies \(q_gU=\phi q_f\).

Because \(f,g\in\mathfrak n^2\), the cotangent spaces of the two quotient
rings are both canonically of dimension \(N+1\):
\[
 \mathfrak m_f/\mathfrak m_f^2\cong
 \mathfrak n/\mathfrak n^2\cong
 \mathfrak m_g/\mathfrak m_g^2.
\]
The map induced by \(\phi\) on cotangent spaces is invertible.  Consequently
the matrix \(L\) of the linear terms of \((u_0,\ldots,u_N)\) is invertible.
The formal inverse function theorem now gives a continuous automorphism
\(U\in\operatorname{Aut}_{\C}(Q)\).

Taking kernels in \(q_gU=\phi q_f\) gives
\[
 U^{-1}((g))=(f),\qquad\text{hence}\qquad (U(f))=(g).
\]
Since \(Q\) is a domain and both ideals are principal, there is a unit
\(v\in Q^*\) such that
\begin{equation}\label{eq:formal-contact}
 U(f)=v g.
\end{equation}
Write \(U(x)=Lx+H(x)\), where every component of \(H\) belongs to
\(\mathfrak n^2\), and write \(v=c+v_+\) with \(c\in\C^*\) and
\(v_+\in\mathfrak n\).  Since \(f\) and \(g\) are homogeneous of degree
\(d\), the degree-\(d\) homogeneous component of the two sides of
\eqref{eq:formal-contact} is
\[
 f(Lx)=c g(x).
\]
Replacing \(L\) by \(L^{-1}\) and rescaling \(c\) gives the form
\(g(Lx)=cf(x)\) in \eqref{eq:ordinary-globalization}.
\end{proof}

\begin{theorem}\label{thm:common-ordinary}
Let \(f,g\in\C[x_0,\ldots,x_{n+1}]_d\) define smooth projective hypersurfaces.
An enhanced equivalence between \(\mathcal C_f\) and \(\mathcal C_g\)
that intertwines the internal grading shifts implies \(X_f\cong X_g\) by a
projective linear transformation.  Under \(\gcd(d,n+2)=1\), every enhanced
equivalence has the required compatibility.
\end{theorem}

\begin{proof}
By Proposition~\ref{prop:common-local}, the equivalence intertwining the internal grading shifts induces
\(\Dsg^{\dgcat}(A_f)\simeq_{\Mor}\Dsg^{\dgcat}(A_g)\), and
Theorem~\ref{thm:hua-keller} gives an isomorphism of complete local
\(\C\)-algebras \(A_f\cong A_g\).  Lemma~\ref{lem:ordinary-globalization}
then produces \(L\in\operatorname{GL}_{n+2}(\C)\) and \(c\in\C^*\) with
\(g(Lx)=cf(x)\).  Therefore the projectivization \([L]\) maps the zero locus of
\(f\) isomorphically onto the zero locus of \(g\).  If
\(\gcd(d,n+2)=1\), Proposition~\ref{prop:intrinsic-shift} gives the required isomorphism intertwining the internal grading shifts
for any enhanced equivalence.
\end{proof}

\section{The Euler-equivariant weighted reconstruction}
\label{sec:weighted-engine}

\subsection{The dual action on the orbit}
For \(\mathcal C_f=\Dsg^{\gr,\dgcat}(R_f)\), define a \(\Gm\)-action on the dg
orbit by
\begin{equation}\label{eq:dual-action}
 \chi_f(\lambda)(h_j)=\lambda^{-j} h_j,
 \qquad h_j\in\Hom_{\mathcal C_f}(E,\rho_f^jF).
\end{equation}
The sign of the exponent is determined by the convention \(M(1)_q=M_{q+1}\).  Since
\((k_j\star h_i)\) has exponent \(i+j\), formula
\eqref{eq:orbit-composition} implies
\(\chi_f(\lambda)(k_j\star h_i)=
\chi_f(\lambda)(k_j)\star\chi_f(\lambda)(h_i)\).  Thus \(\chi_f\) is an
action by dg automorphisms and extends functorially to the perfect dg hull.  The
completed cone carries the Euler action
\begin{equation}\label{eq:euler-action}
 \epsilon_f(\lambda)(x_i)=\lambda^{w_i}x_i.
\end{equation}

\begin{proposition}
\label{prop:equivariant-orbit}
The equivalence in Proposition~\ref{prop:orbit-completion} is equivariant for
the dual action \eqref{eq:dual-action} and the Euler action
\eqref{eq:euler-action}.  An enhanced equivalence intertwining the internal grading shifts induces
an Euler-equivariant equivalence of the completed dg singularity categories.
\end{proposition}

\begin{proof}
For a graded module \(M=\bigoplus_qM_q\), put
\(E_{f,\lambda}=(\epsilon_f(\lambda)^{-1})^*\).  With the grading convention
fixed above, the map
\begin{equation}\label{eq:euler-linearization}
 \alpha_{\lambda,M}:U_fM\longrightarrow E_{f,\lambda}U_fM,
 \qquad \alpha_{\lambda,M}(m)=\lambda^{-q}m\quad(m\in M_q)
\end{equation}
is \(R_f\)-linear.  Conjugation multiplies a map \(h_j:M\to N(j)\) of internal
degree \(j\) by \(\lambda^{-j}\):
\begin{equation}\label{eq:orbit-euler-calculation}
 \alpha_{\lambda,N}\,U_f(h_j)\,\alpha_{\lambda,M}^{-1}
 =\lambda^{-j}U_f(h_j)
 =U_f\bigl(\chi_f(\lambda)(h_j)\bigr).
\end{equation}
Thus the explicit forgetful dg functor is equivariant.  The maximal ideal is
stable under the Euler action, so localization and completion are equivariant.  The comparison
between matrix factorizations and singularity categories is functorial for
automorphisms of the Landau--Ginzburg pair.  The completion functor in~\cite[Theorem~5.7]{Dyckerhoff} is
induced by the natural inclusion of the endomorphism dg algebras of the
stabilized residue field before and after completion; this inclusion is
Euler-equivariant.  Passage to the perfect dg hull is functorial.

Finally, let \((F,\eta)\) be an enhanced equivalence intertwining the
internal grading shifts.  The orbit functor is given by
\eqref{eq:induced-orbit-functor}; it maps the \(j\)-th summand to the \(j\)-th
summand.  Hence
\[
 F_{\mathrm{orb}}\chi_f(\lambda)(h_j)
 =\lambda^{-j}F_{\mathrm{orb}}(h_j)
 =\chi_g(\lambda)F_{\mathrm{orb}}(h_j).
\]
Transport through the equivariant comparison on the two sides gives the stated
Euler-equivariant Morita equivalence.
\end{proof}

\subsection{Equivariance of the Tjurina algebra isomorphism}
We shall use the following equivariant form of the categorical Mather--Yau
isomorphism.

\begin{lemma}
\label{lem:equivariant-HK}
The isomorphism
\[
 \kappa_f:T_f\xrightarrow{\ \cong\ }
 \HH^0\!\left(\Dsg^{\dgcat}(A_f)\right)
\]
is natural for continuous coordinate automorphisms \(\sigma\) satisfying
\(\sigma((f))=(f)\).  In particular it is equivariant for the Euler action.  An
Euler-equivariant Morita equivalence
\(\Dsg^{\dgcat}(A_f)\simeq_{\Mor}\Dsg^{\dgcat}(A_g)\) therefore induces an
isomorphism \(T_f\cong T_g\) of \(\mathbb Z\)-graded algebras.
\end{lemma}

\begin{proof}
Put \(Q=\C[[x_0,\ldots,x_{n+1}]]\) and \(A=Q/(f)\), and denote the dg
singularity category by \(\mathcal D_f=\Dsg^{\dgcat}(A)\).  We use the
Drinfeld-quotient enhancement
\[
 \mathcal D_f=\Db_{\dgcat}(\operatorname{mod}A)/
 \Perf_{\dgcat}(A).
\]
For \(a\in A\), multiplication on a bounded complex \(M\) defines a closed
degree-zero dg natural transformation
\[
 m_a:\id\longrightarrow\id,
 \qquad (m_a)_M(m)=am.
\]
It preserves perfect complexes and therefore descends to \(\mathcal D_f\).
Consequently there is an algebra homomorphism
\begin{equation}\label{eq:multiplication-center-map}
 \mu_f:A\longrightarrow \HH^0(\mathcal D_f),
 \qquad a\longmapsto[m_a].
\end{equation}

We record the factorization of \eqref{eq:multiplication-center-map}.  Under the
dg equivalence between \(\mathcal D_f\) and the matrix-factorization category
over \(Q\), an object is represented by a finite free matrix factorization
\((E,\delta_E)\) with \(\delta_E^2=f\id_E\).  Differentiating this equality
with respect to \(x_i\) gives
\begin{equation}\label{eq:partial-nullhomotopy}
 \delta_E(\partial_i\delta_E)+(\partial_i\delta_E)\delta_E
 =(\partial_i f)\id_E.
\end{equation}
The left side of \eqref{eq:partial-nullhomotopy} is the differential of the odd
endomorphism \(\partial_i\delta_E\) in the matrix-factorization morphism
complex.  In the Hochschild cochain model of
\cite[Theorem~6.6]{Dyckerhoff}, these homotopies assemble to show
that \(\mu_f\) annihilates \((\partial_0f,\ldots,\partial_{n+1}f)\); of course
it also annihilates \((f)\) because \(f=0\) in \(A\).  The induced map
\begin{equation}\label{eq:kappa-multiplication}
 \overline\mu_f:
 \frac{Q}{(f,\partial_0f,\ldots,\partial_{n+1}f)}
 \longrightarrow \HH^0(\mathcal D_f)
\end{equation}
is the isomorphism \(\kappa_f\) of Theorem~\ref{thm:hua-keller}.  More
precisely, in~\cite[Theorems~5.8 and~5.9]{HuaKeller}, the map
\eqref{eq:kappa-multiplication} is factored through singular Hochschild
cohomology via the canonical morphism
\(\HH^*_{\mathrm{sg}}(A,A)\to\HH^*(\mathcal D_f)\); the stable even
cohomology is then identified with the Tjurina algebra.  Each arrow is induced
by a morphism of bar or Tate resolutions and is functorial for algebra
automorphisms.

Let \(\sigma:Q\to Q\) be continuous and satisfy \(\sigma((f))=(f)\); it
induces an automorphism, again denoted \(\sigma\), of \(A\).  Choose a unit
\(v\in Q^*\) such that \(\sigma(f)=vf\).  For \(h\in Q\), define its Jacobian
ideal by
\[
 J_h=(\partial_0h,\ldots,\partial_{n+1}h).
\]
The chain rule gives
\[
 \sigma(J_f)+(f)=J_{\sigma(f)}+(f)=J_{vf}+(f)=J_f+(f),
\]
and the reverse inclusion is obtained from \(\sigma^{-1}\).  Hence \(\sigma\)
also induces an automorphism of \(T_f\).
Define the dg autoequivalence
\(\sigma_*=(\sigma^{-1})^*\) of \(\mathcal D_f\).  On an \(A\)-module \(M\),
the underlying vector space of \(\sigma_*M\) is \(M\), with action
\(a\cdot_{\sigma_*M}m=\sigma^{-1}(a)m\).  Thus, for every \(a\in A\),
componentwise evaluation gives
\begin{equation}\label{eq:conjugation-multiplication}
 \sigma_*\,m_a\,\sigma_*^{-1}=m_{\sigma(a)}.
\end{equation}
Indeed, multiplication by \(a\) on \(\sigma_*^{-1}M\), followed by
restriction through \(\sigma^{-1}\), is multiplication by \(\sigma(a)\) on
\(M\).  Combining \eqref{eq:kappa-multiplication} and
\eqref{eq:conjugation-multiplication} gives the commutative diagram of
algebras below, where \(\operatorname{Ad}(\sigma_*)\) denotes conjugation by
the autoequivalence \(\sigma_*\):
\begin{equation}\label{eq:HK-naturality-square}
\begin{CD}
 T_f @>{\kappa_f}>> \HH^0(\mathcal D_f)\\
 @V{\sigma}VV @VV{\operatorname{Ad}(\sigma_*)}V\\
 T_f @>{\kappa_f}>> \HH^0(\mathcal D_f).
\end{CD}
\end{equation}

Take \(\sigma=\epsilon_f(\lambda)\).  Weighted homogeneity implies
\(\epsilon_f(\lambda)(f)=\lambda^{\deg(f)}f\), so \((f)\) and the Tjurina
ideal are stable.  If \(\bar a\in(T_f)_q\), then
\(\epsilon_f(\lambda)(\bar a)=\lambda^q\bar a\).  Diagram
\eqref{eq:HK-naturality-square} therefore gives
\begin{equation}\label{eq:HK-weight-calculation}
 \operatorname{Ad}(E_{f,\lambda})\bigl(\kappa_f(\bar a)\bigr)
 =\lambda^q\kappa_f(\bar a).
\end{equation}

Now put \(\mathcal D_g=\Dsg^{\dgcat}(A_g)\), let \(\kappa_g\) denote the
analogue of \(\kappa_f\) for \(g\), and let
\(F:\mathcal D_f\simeq_{\Mor}\mathcal D_g\) be Euler-equivariant.
Choose coherent dg natural isomorphisms
\(\theta_\lambda:F E_{f,\lambda}\xrightarrow{\sim}E_{g,\lambda}F\).
Morita invariance of Hochschild cochains gives an algebra isomorphism
\(F_*:\HH^0(\mathcal D_f)\to\HH^0(\mathcal D_g)\).  Conjugating a
representative natural transformation and inserting \(\theta_\lambda\) and
\(\theta_\lambda^{-1}\) yields
\begin{equation}\label{eq:HH-equivariance}
 F_*\operatorname{Ad}(E_{f,\lambda})
 =\operatorname{Ad}(E_{g,\lambda})F_*.
\end{equation}
Hence
\(\kappa_g^{-1}F_*\kappa_f:T_f\to T_g\) is \(\Gm\)-equivariant by
\eqref{eq:HK-weight-calculation} and \eqref{eq:HH-equivariance}.  Since the
isolated-critical-point hypothesis makes \(T_f\) and \(T_g\) finite-dimensional
rational \(\Gm\)-modules, they decompose as direct sums of their character
spaces.  The equivariant algebra isomorphism maps \((T_f)_q\) onto
\((T_g)_q\) for every \(q\in\mathbb Z\), and is therefore an isomorphism of
graded algebras.
\end{proof}

\subsection{Reconstruction of weighted projective hypersurfaces}
Weighted homogeneity gives the Euler identity
\begin{equation}\label{eq:euler-identity}
 \deg(f)\,f=\sum_{i=0}^{n+1}w_i x_i\partial_i f.
\end{equation}
Since the characteristic is zero, \(f\) lies in its Jacobian ideal, and therefore
\begin{equation}\label{eq:tjurina-milnor}
 T_f\cong
 \frac{P}{(\partial_0f,\ldots,\partial_{n+1}f)}
 =:\Jac(f)
\end{equation}
as weighted graded algebras.  The completed and polynomial quotients agree
because the isolated critical point makes them finite-dimensional.

\begin{proposition}\label{prop:weighted-globalization}
The following consequence of~\cite[Corollary~3]{Ahmed} will be used below.  Let
\(f,g\in P_d\) be weighted homogeneous polynomials with isolated critical
points, for the same positive weights and degree.  If \(T_f\cong T_g\) as
graded algebras, there are a weight-preserving polynomial automorphism \(u\) of
\(P\) and \(c\in\C^*\) such that
\begin{equation}\label{eq:weighted-globalization}
 u(g)=cf.
\end{equation}
\end{proposition}

\begin{proof}
Equation~\eqref{eq:tjurina-milnor} identifies the Tjurina algebras with the
graded Milnor algebras.  By~\cite[Corollary~3]{Ahmed}, an isomorphism of
these graded algebras implies right equivalence by a change of coordinates
preserving the weights.  Each coordinate of weight
\(w_i\) is a weighted homogeneous power series of that weight.  Positivity of
the weights leaves only finitely many monomials of a fixed weight, so the
coordinate change and its inverse are polynomial.  This gives a
weight-preserving polynomial automorphism; allowing multiplication by a
nonzero scalar gives the form stated in~\eqref{eq:weighted-globalization}.
\end{proof}

\begin{theorem}\label{thm:common-weighted}
Let \(f,g\in P_d\) define quasismooth weighted hypersurface stacks for the same weights and
degree, and assume that neither presentation is a linear cone.  An enhanced
equivalence \(\mathcal C_f\simeq_{\Mor}\mathcal C_g\) intertwining the
internal grading shifts implies that a
weight-preserving polynomial automorphism sends \(g\) to a nonzero scalar
multiple of \(f\).  Under \(\gcd(d,W)=1\), every enhanced equivalence is
automatically compatible with the internal grading shifts.
\end{theorem}

\begin{proof}
Let \((F,\eta)\) be the enhanced equivalence intertwining the internal
grading shifts.  Formula
\eqref{eq:induced-orbit-functor} gives an equivalence of dg orbits, and
Proposition~\ref{prop:equivariant-orbit} identifies its perfect hull with an
Euler-equivariant Morita equivalence
\[
 \Dsg^{\dgcat}(A_f)\simeq_{\Mor}\Dsg^{\dgcat}(A_g).
\]
By Lemma~\ref{lem:equivariant-HK}, the induced map on \(\HH^0\), conjugated by
\(\kappa_f\) and \(\kappa_g\), is an isomorphism
\(T_f\xrightarrow{\sim}T_g\) in the category of graded algebras.
Proposition~\ref{prop:weighted-globalization} therefore supplies a
weight-preserving \(u\in\operatorname{Aut}_{\C}(P)\) and \(c\in\C^*\) with
\(u(g)=cf\).  Since \(u(\lambda^{w_i}x_i)=\lambda^{w_i}u(x_i)\) coordinatewise,
\(u\) is \(\Gm\)-equivariant on \(\mathbb A^{n+2}\setminus\{0\}\), and hence
descends to an isomorphism of the quotient hypersurface stacks.  If
\(\gcd(d,W)=1\), Proposition~\ref{prop:intrinsic-shift} constructs \(\eta\)
from an arbitrary enhanced equivalence.
\end{proof}

\section{Fano hypersurfaces}\label{sec:fano}

\subsection{Passage from the Kuznetsov component}
Assume \(a=W-d>0\).  By
\cite[Theorems~2.5(i) and~3.10]{Orlov}, there is a semiorthogonal decomposition
of the Fano hypersurface \(X_f\),
\begin{equation}\label{eq:orlov-fano}
 \Db(X_f)=
 \langle\mathcal{K}u(X_f),\mathcal O_{X_f},\mathcal O_{X_f}(1),\ldots,
 \mathcal O_{X_f}(a-1)\rangle.
\end{equation}
The equivalence for equivariant factorizations in
\cite[Theorem~6.13]{BFK} identifies the Kuznetsov component with the full
graded matrix-factorization category:
\[
 \mathcal{K}u(X_f)_{\dgcat}\simeq_{\Mor}
 \mathcal C_f=\MF^{\gr,\dgcat}(P,f).
\]
Under this equivalence, the geometric rotation
\[
 \rho_f(F)=\mathbf L_{\mathcal O_{X_f}}
 \bigl(F\otimes\mathcal O_{X_f}(1)\bigr)
\]
corresponds to the internal grading shift
\((1):\mathcal C_f\to\mathcal C_f\).  Transporting
\eqref{eq:periodicity} and \eqref{eq:serre-full} through this equivalence gives
the following identities, where \(S_f\) denotes the Serre functor of the
Kuznetsov component \(\mathcal{K}u(X_f)\), and \(S_g\) has the analogous
meaning for \(X_g\):
\begin{equation}\label{eq:fano-identities}
 \rho_f^d\simeq[2],\qquad
 S_f\simeq\rho_f^{-a}[n].
\end{equation}

If \(\gcd(d,W)=1\), choose \(r,s\in\mathbb Z\) with \(ra+sd=1\).  Then
\begin{equation}\label{eq:fano-intrinsic}
 \rho_f\simeq S_f^{-r}[rn+2s].
\end{equation}
Thus an arbitrary enhanced equivalence of Kuznetsov components preserves
rotation because it preserves the Serre functor.

\subsection{Ordinary Fano hypersurfaces}
Let \(f\in\C[x_0,\ldots,x_{n+1}]\) be homogeneous of degree \(d\), assume
\(2\le d<n+2\), and put
\[
 \mathbb P=\mathbb P^{n+1},\qquad
 R_f=\C[x_0,\ldots,x_{n+1}]/(f),\qquad a=n+2-d.
\]
If \(X_f\subset\mathbb P\) is smooth, its affine cone is smooth away
from the vertex.  Since \(d\ge2\), the vertex is singular, so
Proposition~\ref{prop:orbit-completion} applies.

\begin{proposition}
\label{prop:fano-periodicity-serre}
There are isomorphisms of enhanced functors
\[
 \rho_f^d\cong[2],\qquad S_f\cong\rho_f^{-a}[n].
\]
\end{proposition}

\begin{proof}
Let \(\Psi_f:\mathcal{K}u(X_f)_{\dgcat}\simeq_{\Mor}
\MF^{\gr,\dgcat}(P,f)\) be the equivalence of
\cite[Theorems~2.5(i) and~3.10]{Orlov}, with the dg enhancement provided by
\cite[Theorem~6.13]{BFK}, normalized by
\(\Psi_f\rho_f\simeq(1)\Psi_f\).  A graded factorization is a diagram
\[
 E_0\xrightarrow{\delta_0}E_1
 \xrightarrow{\delta_1}E_0(d).
\]
The dg translation interchanges \(E_0,E_1\) and changes one differential by a
sign.  Applying it twice returns the original differentials and shifts every
graded module by \(d\); hence \([2]\simeq(d)=(1)^d\) as dg functors.  Transport
by \(\Psi_f^{-1}\) gives \(\rho_f^d\simeq[2]\).

The canonical bundle of \(\mathbb P\) is
\(\omega_{\mathbb P}\cong\mathcal O_{\mathbb P}(-(n+2))\), and adjunction
gives
\[
 \omega_{X_f}\cong\mathcal O_{X_f}(d-n-2)
 =\mathcal O_{X_f}(-a).
\]
Equivalently, the graded canonical module of \(R_f\) is
\(\omega_{R_f}\cong R_f(d-n-2)=R_f(-a)\).  Graded Auslander--Reiten duality on the
stable category of maximal Cohen--Macaulay \(R_f\)-modules gives, for two such
modules \(M\) and \(N\), the functorial quasi-isomorphism
\[
 \RHom_{R_f}(M,N)^\vee
 \simeq \RHom_{R_f}\bigl(N,M(-a)[n]\bigr).
\]
Thus the Serre functor on \(\Dsg^{\gr,\dgcat}(R_f)\) is \((-a)[n]\).  Transport
through \(\Psi_f\) yields \(S_f\simeq\rho_f^{-a}[n]\).
\end{proof}

\begin{lemma}
\label{lem:fano-intrinsic}
If \(\gcd(d,n+2)=1\), choose integers \(r,s\) satisfying
\(r(n+2-d)+sd=1\).  Then
\[
 \rho_f\cong S_f^{-r}[rn+2s].
\]
In particular every enhanced equivalence of the Kuznetsov components
intertwines their rotations.
\end{lemma}

\begin{proof}
Since \(\gcd(d,n+2-d)=\gcd(d,n+2)=1\), such integers exist.  Using
Proposition~\ref{prop:fano-periodicity-serre},
\[
 \rho_f
 \cong\rho_f^{r(n+2-d)+sd}
 \cong(\rho_f^a)^r(\rho_f^d)^s
 \cong(S_f^{-1}[n])^r[2s]
\cong S_f^{-r}[rn+2s].
\]
Let \(\Phi:\mathcal{K}u(X_f)_{\dgcat}\to\mathcal{K}u(X_g)_{\dgcat}\) be an enhanced
equivalence.  Transport of the Serre pairing gives
\(\Phi S_f\simeq S_g\Phi\).  Therefore
\[
 \Phi\rho_f
 \simeq\Phi S_f^{-r}[rn+2s]
 \simeq S_g^{-r}[rn+2s]\Phi
 \simeq\rho_g\Phi,
\]
which is the required isomorphism of dg functors intertwining the rotations.
\end{proof}

\begin{remark}
If \(e=\gcd(d,n+2)>1\), the same calculation recovers only \(\rho_f^e\).
The Serre functor does not choose an \(e\)-th root.  Therefore
Theorem~\ref{thm:rotation-reconstruction} remains valid, but compatibility of an enhanced equivalence with the rotation functors
requires an additional argument.  Accordingly, the theorem in~\cite[Theorem~1.2]{Pirozhkov} is formulated using the
rotation functor, while the Hochschild--Serre method can apply under different
numerical assumptions.
\end{remark}

\begin{proposition}
\label{prop:fano-completed-cone}
Let \(X_f\) and \(X_g\) be smooth degree-\(d\) hypersurfaces in
\(\mathbb P^{n+1}\).  An enhanced equivalence
\(\Phi:\mathcal{K}u(X_f)\simeq_{\Mor}\mathcal{K}u(X_g)\)
intertwining the rotation functors induces a Morita equivalence
\[
 \Dsg^{\dgcat}(A_f)\simeq_{\Mor}\Dsg^{\dgcat}(A_g).
\]
Consequently \(A_f\cong A_g\).
\end{proposition}

\begin{proof}
Choose the equivalences of
\cite[Theorems~2.5(i) and~3.10]{Orlov}, with the dg enhancements
of~\cite[Theorem~6.13]{BFK},
\(\Psi_f:\mathcal{K}u(X_f)_{\dgcat}\simeq_{\Mor}\mathcal C_f\) and
\(\Psi_g:\mathcal{K}u(X_g)_{\dgcat}\simeq_{\Mor}\mathcal C_g\), normalized by
\(\Psi_f\rho_f\simeq(1)\Psi_f\) and similarly for \(g\).  Then
\[
 F:=\Psi_g\Phi\Psi_f^{-1}:\mathcal C_f\simeq_{\Mor}\mathcal C_g
\]
admits an isomorphism \(F(1)\simeq(1)F\).  Applying the explicit construction
\eqref{eq:induced-orbit-functor} gives
\[
 \Perf_{\dgcat}(\mathcal C_f/(1))
 \simeq_{\Mor}\Perf_{\dgcat}(\mathcal C_g/(1)).
\]
The two instances of Proposition~\ref{prop:orbit-completion} identify the left
and right sides with \(\Dsg^{\dgcat}(A_f)\) and
\(\Dsg^{\dgcat}(A_g)\), respectively.

Smoothness of \(X_f\) and \(X_g\) implies that the two cone singularities are
isolated.  Both complete local rings have Krull dimension \(n+1\) and embedding
dimension \(n+2\).  Hence Theorem~\ref{thm:hua-keller} applies without a
Kn\"orrer stabilization and gives \(A_f\cong A_g\).  Equivalently,
Theorem~\ref{thm:kalck} permits only the addition of an even number of quadratic
variables; equality of embedding dimensions forces that number to be zero.
\end{proof}

\subsection{Weighted Fano hypersurfaces}
Fix positive weights \(w_0,\ldots,w_{n+1}\), put
\[
 P=\C[x_0,\ldots,x_{n+1}],\qquad \deg x_i=w_i,
 \qquad W=\sum_iw_i,
\]
and, for a weighted homogeneous \(f\in P_d\), define the hypersurface stack
\[
 X_f=[(V(f)\setminus\{0\})/\Gm]\subset\mathbb P(w_0,\ldots,w_{n+1}).
\]
The theorem for weighted Fano hypersurfaces will be used under the following
hypotheses:
\begin{enumerate}[label=\textup{(W\arabic*)},leftmargin=3em]
\item the affine cone is smooth away from the origin, equivalently \(f\) is
quasismooth;
\item \(f\in(x_0,\ldots,x_{n+1})^2\), so the presentation is not a linear
cone;
\item \(a=W-d>0\), so \(X_f\) is Fano;
\item the weight system and \(d\) are fixed when two hypersurfaces are
compared.
\end{enumerate}
The first two conditions make \(A_f\) a nontrivial isolated hypersurface
singularity.  The hypersurface stack is smooth, while its coarse space may
retain quotient singularities.  Unless stated otherwise, the categorical
assertions concern the stack.  Let \(\rho_f\) be the grading shift transported from
\(\Dsg^{\gr}(R_f)\) to \(\mathcal{K}u(X_f)\) by the equivalence in
\cite[Theorems~2.5(i) and~3.10]{Orlov}.

\begin{proposition}
There are isomorphisms of enhanced functors
\[
 \rho_f^d\cong[2],\qquad S_f\cong\rho_f^{-a}[n].
\]
If \(\gcd(d,W)=1\) and \(ra+sd=1\), then
\[
 \rho_f\cong S_f^{-r}[rn+2s].
\]
\end{proposition}

\begin{proof}
The equivalence in~\cite[Theorems~2.5(i) and~3.10]{Orlov}, with its dg
enhancement in~\cite[Theorem~6.13]{BFK}, identifies
\(\rho_f\) with \((1)\).  The
calculation in the matrix-factorization model gives \((1)^d\simeq[2]\).
Furthermore, the canonical bundle of
\(\mathbb P(w_0,\ldots,w_{n+1})\) is
\(\mathcal O_{\mathbb P(w_0,\ldots,w_{n+1})}(-W)\).  Adjunction, equivalently
the graded hypersurface resolution, gives
\(\omega_{R_f}\cong R_f(d-W)=R_f(-a)\).  Graded
Auslander--Reiten duality therefore
identifies the Serre functor with \((-a)[n]\), and transport through the same
equivalence gives \(S_f\simeq\rho_f^{-a}[n]\).

Since \(\gcd(a,d)=\gcd(W-d,d)=\gcd(W,d)=1\), choose \(r,s\) with
\(ra+sd=1\).  In the group of enhanced autoequivalences,
\begin{align*}
 S_f^{-r}[rn+2s]
 &\simeq(\rho_f^{-a}[n])^{-r}[rn+2s]\\
 &\simeq\rho_f^{ra}[2s]
 \simeq\rho_f^{ra+sd}=\rho_f.
\end{align*}
If \((r',s')=(r+kd,s-ka)\), then composing the expression for \((r',s')\) with the inverse of that
for \((r,s)\) gives
\[
 S_f^{-kd}[kdn-2ka]
 \simeq \rho_f^{kad}[-kdn+kdn-2ka]
 \simeq \rho_f^{kad}[ -2ka]
 \simeq \id,
\]
where \([2ka]\simeq\rho_f^{kad}\).  Thus different choices of the B\'ezout coefficients give
isomorphic enhanced functors.
\end{proof}

The coprimality condition is not a hypothesis of
Theorems~\ref{thm:hua-keller} or~\ref{thm:kalck}.  It is used earlier and only
to recover the degree shift from the abstract Kuznetsov component.  If an
equivalence is assumed to commute with \(\rho\), the weighted argument works
without coprimality.

\subsection{Proofs for Fano hypersurfaces}
\begin{proof}[Proof of Theorem~\ref{thm:rotation-reconstruction}]
Proposition~\ref{prop:fano-completed-cone} first gives a Morita equivalence of
the completed singularity categories, and Theorem~\ref{thm:hua-keller} then
gives an isomorphism \(A_f\cong A_g\).  Applying
Lemma~\ref{lem:ordinary-globalization} to this
complete local algebra isomorphism yields \(L\in\operatorname{GL}_{n+2}(\C)\)
and \(c\in\C^*\) satisfying \(g(Lx)=cf(x)\).  Hence
\([L](V(f))=V(g)\).  The completed singularity is determined before the reconstruction of the projective hypersurface, as asserted.
\end{proof}

\begin{proof}[Proof of Corollary~\ref{thm:fano-ordinary}]
Let \(\Phi:\mathcal{K}u(X_f)_{\dgcat}\simeq_{\Mor}\mathcal{K}u(X_g)_{\dgcat}\).  By
Lemma~\ref{lem:fano-intrinsic}, the fixed B\'ezout pair \((r,s)\) gives
\[
 \Phi\rho_f\simeq\Phi S_f^{-r}[rn+2s]
 \simeq S_g^{-r}[rn+2s]\Phi\simeq\rho_g\Phi.
\]
Thus \(\Phi\) satisfies the hypothesis of
Theorem~\ref{thm:rotation-reconstruction}, which gives the claimed equation and
projective isomorphism.
\end{proof}

\begin{proof}[Proof of Theorem~\ref{thm:fano-weighted}]
Let \(\Phi:\mathcal{K}u(X_f)_{\dgcat}\simeq_{\Mor}\mathcal{K}u(X_g)_{\dgcat}\).  The weighted
equivalences in \eqref{eq:orlov-fano} transport \(\Phi\) to
\(F:\mathcal C_f\simeq_{\Mor}\mathcal C_g\).  The preceding proposition and
the same B\'ezout coefficients give \(F\rho_f\simeq\rho_gF\).  All hypotheses of
Theorem~\ref{thm:common-weighted} hold by quasismoothness and the assumption that neither presentation is a linear cone.  It follows that there are
\(u\in\operatorname{Aut}_{\gr}(P)\) and \(c\in\C^*\) with \(u(g)=cf\), and the
induced \(\Gm\)-equivariant automorphism of the punctured affine space descends
to \(X_f\cong X_g\).
\end{proof}

\begin{proof}[Proof of Corollary~\ref{cor:index-one}]
For \(d=n+1\) in \(\mathbb P^{n+1}\), the Fano index is \(a=1\), and the
decomposition in~\cite[Theorems~2.5(i) and~3.10]{Orlov} is
\[
 \Db(X_f)=\langle\mathcal{K}u(X_f),\mathcal O_{X_f}\rangle,
 \qquad \mathcal{K}u(X_f)=\mathcal O_{X_f}^{\perp}.
\]
Moreover \(\gcd(n+1,n+2)=1\), and \eqref{eq:fano-intrinsic} becomes
\[
 \rho_f\simeq S_{\mathcal{K}u(X_f)}^{-1}[n].
\]
If \(\Phi:\mathcal O_{X_f}^{\perp}\simeq_{\Mor}
\mathcal O_{X_g}^{\perp}\), then transport of the Serre bimodule gives
\(\Phi S_f\simeq S_g\Phi\), and consequently
\[
 \Phi\rho_f\simeq\Phi S_f^{-1}[n]
 \simeq S_g^{-1}[n]\Phi\simeq\rho_g\Phi.
\]
Theorem~\ref{thm:rotation-reconstruction} gives \(X_f\cong X_g\), and the
orbit construction in Proposition~\ref{prop:fano-completed-cone} gives
\(A_f\cong A_g\).

To identify the complete local ring among all such singularities, let \(S\) be any complete local
commutative Noetherian \(\C\)-algebra whose dg singularity category is Morita
equivalent to the category obtained from the dg orbit of
\(\mathcal O_{X_f}^{\perp}\).  Theorem~\ref{thm:kalck} gives a hypersurface
presentation of \(S\) and, after ordering the presentations by
embedding dimension, an integer \(r\ge0\) such that its equation is formally
right equivalent to
\[
 f+z_1^2+\cdots+z_{2r}^2.
\]
If \(\dim S=\dim A_f=n+1\), the two hypersurface presentations have the same
embedding dimension, hence \(2r=0\), and \(S\cong A_f\).
\end{proof}

\begin{remark}
The polarized projective hypersurface determines its affine cone.  The
categorical argument first recovers the formal germ at the vertex.
Theorem~\ref{thm:kalck} identifies this germ among complete local commutative
singularities, up to the addition of an even number of nondegenerate
quadratic variables.  Comparing the lowest-degree terms of the homogeneous
equations then recovers the projective hypersurface.
\end{remark}

\begin{proof}[Proof of Corollary~\ref{cor:double-covers}]
Write the double cover as
\[
 X_F=\{z^2-F_{2n}(x_0,\ldots,x_n)=0\}
 \subset\mathbb P(1^{n+1},n).
\]
Here \(W=2n+1\), \(d=2n\), and \(\gcd(d,W)=1\), so
Theorem~\ref{thm:fano-weighted} applies.  A weight-preserving automorphism has
\[
 x\longmapsto Lx,\qquad z\longmapsto\alpha z+h_n(x),
\]
with \(L\in\operatorname{GL}_{n+1}(\C)\) and \(\alpha\ne0\).  If it carries
\(z^2-G(x)\) to a scalar multiple of \(z^2-F(x)\), then for some
\(c\in\C^*\)
\[
 (\alpha z+h_n(x))^2-G(Lx)=c\bigl(z^2-F(x)\bigr).
\]
Comparison of the \(z^2\)-coefficient gives \(c=\alpha^2\), while comparison
of the \(z\)-coefficient gives \(2\alpha h_n=0\).  Since the characteristic is
zero and \(\alpha\ne0\), one has \(h_n=0\).  The term independent of \(z\)
then gives \(G(Lx)=\alpha^2F(x)\).  Thus \([L]\) carries \(V(F)\) to \(V(G)\),
and the corresponding equivariant map of the two weighted hypersurfaces is an
isomorphism of the double covers.
\end{proof}

For comparison with the local reconstruction, the Tjurina algebra can be
computed directly:
\[
 T_{z^2-F}
 =\frac{\C[x_0,\ldots,x_n,z]}
 {(z^2-F,2z,\partial_0F,\ldots,\partial_nF)}
 \cong
 \frac{\C[x_0,\ldots,x_n]}
 {(\partial_0F,\ldots,\partial_nF)}.
\]
The normalized Euler identity removes \(F\) in the last quotient.  Thus the
equivariant local algebra records exactly the graded Jacobian algebra of the
branch equation.

\begin{remark}
Both the ordinary and weighted arguments pass through the full category
\(\mathcal C_f\) and use the same comparison between the dg orbit and the singularity category of the completed cone as the
argument for hypersurfaces of general type.  No separate calculation of
the components of the Hochschild--Serre algebra corresponding to the Jacobian algebra is needed.  For a weighted
hypersurface one must retain, in addition, the dual torus action, which corresponds to the Euler action.
\end{remark}

\section{Hypersurfaces of general type}\label{sec:general-type}

\subsection{The derived category and the graded singularity category}
Assume \(b=d-W>0\).  By~\cite[Theorem~2.5(ii), Proposition~2.17,
and Theorem~3.10]{Orlov}, there is a fully faithful functor from
\(\Db(X_f)\) to \(\mathcal C_f\) whose orthogonal is generated by \(b\)
exceptional objects.  For ordinary hypersurfaces, this is precisely
\cite[Theorem~3.11(ii)]{Orlov}; the weighted case is explained
in~\cite[Remark~3.12]{Orlov}.
Thus \(\Db(X_f)\) is an admissible component of \(\mathcal C_f\).  In
particular, the
theorem for a hypersurface of general type must start with
\[
 \mathcal C_f=\MF^{\gr,\dgcat}(P,f),
\]
equipped with the internal grading shift
\(\rho_f=(1):\mathcal C_f\to\mathcal C_f\) and its dg orbit.

The categorical identities now read
\begin{equation}\label{eq:general-identities}
 \rho_f^d\simeq[2],\qquad
 S_{\mathcal C_f}\simeq\rho_f^b[n].
\end{equation}
If \(\gcd(d,W)=1\) and \(ub+vd=1\), then
\begin{equation}\label{eq:general-intrinsic}
 \rho_f\simeq S_{\mathcal C_f}^{u}[2v-un].
\end{equation}
This is formula~\eqref{eq:intrinsic-shift} with \(d-W>0\).

\begin{proposition}
\label{prop:general-serre}
On the full graded matrix-factorization category one has
\[
 S_{\mathcal C_f}\cong\rho_f^{d-W}[n].
\]
If \(\gcd(d,W)=1\), choose \(u,v\in\mathbb Z\) with
\(u(d-W)+vd=1\).  Then
\[
 \rho_f\cong S_{\mathcal C_f}^{u}[2v-un].
\]
Consequently every dg Morita equivalence
\(F:\mathcal C_f\simeq_{\Mor}\mathcal C_g\), for the same weights and degree,
automatically intertwines the internal shifts:
\(F\rho_f\cong\rho_gF\).
\end{proposition}

\begin{proof}
Put
\(\mathbb P_w=\mathbb P(w_0,\ldots,w_{n+1})\).  Its canonical bundle is
\(\omega_{\mathbb P_w}\cong\mathcal O_{\mathbb P_w}(-W)\), and adjunction
gives \(\omega_{X_f}\cong\mathcal O_{X_f}(d-W)\).  For the graded coordinate ring, applying \(\operatorname{Ext}^1_P(-,P(-W))\) to
\[
 0\longrightarrow P(-d)\xrightarrow{\cdot f}P
 \longrightarrow R_f\longrightarrow0
\]
gives \(\omega_{R_f}\cong R_f(d-W)\).  Since \(\dim R_f=n+1\), graded
Auslander--Reiten duality gives, for objects \(M,N\) of the graded singularity
category, functorial quasi-isomorphisms
\[
 \RHom_{R_f}(M,N)^\vee
 \simeq
 \RHom_{R_f}\bigl(N,M(d-W)[n]\bigr)
\]
in the graded singularity category.  Hence
\(S_{\mathcal C_f}\simeq\rho_f^{d-W}[n]\).

Let \(u(d-W)+vd=1\).  The periodicity \(\rho_f^d\simeq[2]\) gives
\begin{align*}
 S_{\mathcal C_f}^{u}[2v-un]
 &\simeq(\rho_f^{d-W}[n])^u[2v-un]\\
 &\simeq\rho_f^{u(d-W)}[2v]
 \simeq\rho_f^{u(d-W)+vd}=\rho_f.
\end{align*}
If \(F:\mathcal C_f\simeq_{\Mor}\mathcal C_g\), transport of the Serre
bimodule gives
\(F S_{\mathcal C_f}\simeq S_{\mathcal C_g}F\); applying \(F\) to the
displayed intrinsic formula gives \(F\rho_f\simeq\rho_gF\).  Finally
\(\gcd(d,d-W)=\gcd(d,W)\), so such integers exist under the stated
coprimality hypothesis.
\end{proof}

\begin{corollary}
\label{cor:general-local}
Under the hypotheses of Proposition~\ref{prop:general-serre}, every Morita
equivalence \(F:\mathcal C_f\simeq_{\Mor}\mathcal C_g\) induces a Morita
equivalence
\[
 \Dsg^{\dgcat}(A_f)\simeq_{\Mor}\Dsg^{\dgcat}(A_g).
\]
\end{corollary}

\begin{proof}
Proposition~\ref{prop:general-serre} gives an enhanced natural isomorphism
\(\eta:F\rho_f\xrightarrow{\sim}\rho_gF\).  Its coherent iterates define the
dg orbit functor by \eqref{eq:induced-orbit-functor}.  Applying the same
construction to a quasi-inverse of \(F\) proves
\[
 \Perf_{\dgcat}(\mathcal C_f/\rho_f)
 \simeq_{\Mor}\Perf_{\dgcat}(\mathcal C_g/\rho_g).
\]
The two equivalences of
Proposition~\ref{prop:orbit-completion} identify these terms with
\(\Dsg^{\dgcat}(A_f)\) and \(\Dsg^{\dgcat}(A_g)\), respectively.
\end{proof}

\subsection{Proofs for hypersurfaces of general type}
\begin{proof}[Proof of Theorem~\ref{thm:general-ordinary}]
Let \(F:\mathcal C_f\simeq_{\Mor}\mathcal C_g\) be the equivalence in the
hypothesis.  Since \(W=n+2\), Proposition~\ref{prop:general-serre} and the
coprimality hypothesis give \(F\rho_f\simeq\rho_gF\).  Corollary
\ref{cor:general-local} and Theorem~\ref{thm:hua-keller} give, respectively,
\[
 \Dsg^{\dgcat}(A_f)\simeq_{\Mor}\Dsg^{\dgcat}(A_g),
 \qquad A_f\cong A_g.
\]
The two complete rings have equal Krull and embedding dimensions, so no
Kn\"orrer stabilization occurs.  Lemma~\ref{lem:ordinary-globalization}
applied to \(A_f\cong A_g\) yields
\(L\in\operatorname{GL}_{n+2}(\C)\) and \(c\in\C^*\) such that
\(g(Lx)=cf(x)\).  Therefore \([L]\) restricts to an isomorphism
\(X_f\xrightarrow{\sim}X_g\).
\end{proof}

\begin{proof}[Proof of Theorem~\ref{thm:general-weighted}]
Let \(F:\mathcal C_f\simeq_{\Mor}\mathcal C_g\).  Proposition
\ref{prop:general-serre} produces an enhanced isomorphism
\(F\rho_f\simeq\rho_gF\).  The induced orbit functor
\eqref{eq:induced-orbit-functor} preserves every exponent \(j\), so
Proposition~\ref{prop:equivariant-orbit} gives an Euler-equivariant Morita
equivalence
\[
 \Dsg^{\dgcat}(A_f)\simeq_{\Mor}\Dsg^{\dgcat}(A_g).
\]
By Lemma~\ref{lem:equivariant-HK}, the induced \(\HH^0\)-isomorphism is an
isomorphism \(T_f\cong T_g\) of graded algebras.  The identities
\[
 \deg(f)\,f=\sum_iw_ix_i\partial_i f,
 \qquad
 \deg(g)\,g=\sum_iw_ix_i\partial_i g
\]
identify these with the graded Milnor algebras.  Proposition
\ref{prop:weighted-globalization} gives
\(u\in\operatorname{Aut}_{\gr}(P)\) and \(c\in\C^*\) with \(u(g)=cf\).
The induced \(\Gm\)-equivariant automorphism of the punctured affine space
descends to an isomorphism of quotient hypersurface stacks and, functorially,
to an isomorphism of their coarse spaces.
\end{proof}

For a weighted hypersurface, no new local reconstruction theorem is needed.
We use the equivariance of the isomorphism between the Tjurina algebra
and degree-zero Hochschild cohomology: the dual
action on the orbit becomes the Euler action on \(A_f\), and naturality of the
multiplication classes makes \(T_f\cong T_g\) an isomorphism of graded
algebras.  The weighted conclusion is then obtained from
Theorem~\ref{thm:hua-keller} and the homogeneous reconstruction result
\cite[Corollary~3]{Ahmed}.

\section{Comparison with the Hochschild--Serre method}
\label{sec:comparison}

\subsection{The grading shift and the Hochschild--Serre algebra}
For a smooth proper dg category \(\mathcal C\) with Serre functor \(S\), set
\[
 \HS^{p,q}(\mathcal C)=
 \Hom_{\operatorname{Fun}(\mathcal C,\mathcal C)}
     (\id_{\mathcal C},S^q[p]),
 \qquad
 \HS(\mathcal C)=\bigoplus_{p,q\in\mathbb Z}\HS^{p,q}(\mathcal C),
\]
with multiplication by twisted composition.  Under the assumption
\(\gcd(d,W)=1\),
\(\rho\) is a composition of powers of \(S\) and \([1]\), so the spaces involving
powers of \(\rho\) lie on an intrinsic diagonal of \(\HS(\mathcal C)\).  Both
methods therefore use the same categorical polarization and both use the associated
multiplicative structure.

For example, suppose \(d<W\), put \(a=W-d\), and choose integers \(r,s\)
satisfying \(ra+sd=1\).  Then \(\rho\cong S^{-r}[rn+2s]\), and hence
\begin{equation}\label{eq:HS-diagonal}
 \Hom_{\operatorname{Fun}(\mathcal C,\mathcal C)}
 \bigl(\id,\rho^m[\ell]\bigr)
 =\HS^{m(rn+2s)+\ell,-rm}(\mathcal C).
\end{equation}
The composition on the left is exactly the twisted composition inherited from
\eqref{eq:HS-product}.  Thus both methods use these composition products, which contain more
information than the dimensions of the corresponding vector spaces.

The Hochschild--Serre method selects bidegrees and calculates them through
equivariant matrix-factorization kernels.  For suitable bidegrees, these calculations recover the
Jacobian algebra or a specified Veronese subalgebra
\cite{BFK,LinZhang,LinRennemoZhang}.  The present method uses all
powers of \(\rho\) to form the dg orbit, identifies its perfect hull with the
completed singularity category, and invokes
\(\HH^0(\Dsg^{\dgcat}(A_f))\cong T_f\).  For a weighted hypersurface, the dual action
on the orbit retains the grading that the Hochschild--Serre bigrading records.

The methods are complementary.  If \(\gcd(d,W)>1\), the Serre functor does not
choose a root \(\rho\), while a suitable Hochschild--Serre diagonal may still
recover a Veronese subalgebra sufficient for reconstruction.  Conversely, once \(\rho\) is intrinsic
or included in the data, the singularity-category method gives a short
reconstruction and determines the formal
singularity type.

\subsection{The index-one orthogonal category}
For a degree-\((n+1)\) hypersurface in \(\mathbb P^{n+1}\), the identification
of the right orthogonal with the Kuznetsov component,
\(\mathcal O_X^\perp=\mathcal{K}u(X)\), together with the formula
\(\rho=S^{-1}[n]\), shows that the
dg orbit is constructed from \(\mathcal O_X^\perp\) alone.  In
Hochschild--Serre language, all spaces
\[
 \Hom(\id,\rho^m[\ell])
 =\HS^{mn+\ell,-m}(\mathcal O_X^\perp)
\]
lie on an intrinsic diagonal.  The Hochschild--Serre method reconstructs the
Jacobian algebra from suitable pieces of this diagonal~\cite{LinZhang}; the
present method takes the orbit determined by the same diagonal and invokes
Theorem~\ref{thm:hua-keller}.  Thus the proof for Fano hypersurfaces uses the same intrinsic grading
shift as the earlier Hochschild--Serre argument and reduces reconstruction
to the dg singularity category.

\end{document}